\expandafter\ifx\csname pdfpkresolution\endcsname\relax
\documentclass[11pt,draft]{amsart}
\else
\documentclass[11pt]{amsart}
\fi

\usepackage[all]{xy}
\usepackage{amssymb,xspace,graphicx}
\usepackage{epigraph}
\newcommand*{\Aa}{\ensuremath{\mathbb A}\xspace}

\newcommand*{\Bc}{\ensuremath{\mathcal B}\xspace}
\newcommand*{\C}{\ensuremath{\mathbb C}\xspace}

\newcommand*{\N}{\ensuremath{\mathbb N}\xspace}
\newcommand*{\Ooo}{\ensuremath{\mathcal O}\xspace}
\newcommand*{\PP}{\ensuremath{\mathbb P}\xspace}

\newcommand*{\Rs}{\ensuremath{\mathcal R}\xspace}
\newcommand*{\U}{\ensuremath{\mathcal U}\xspace}
\newcommand*{\X}{\ensuremath{\mathcal X}\xspace}
\newcommand*{\Y}{\ensuremath{\mathcal Y}\xspace}
\newcommand*{\Z}{\ensuremath{\mathbb Z}\xspace}

\DeclareMathOperator{\Aut}{Aut}
\DeclareMathOperator{\Br}{Br}
\DeclareMathOperator{\br}{br}

\DeclareMathOperator{\HOM}{\underline{Hom}}
\DeclareMathOperator{\im}{Im}
\DeclareMathOperator{\Ker}{Ker}
\DeclareMathOperator{\Mon}{Mon}
\DeclareMathOperator{\Pic}{Pic}
\DeclareMathOperator{\PSL}{PSL}
\DeclareMathOperator{\pr}{pr}
\DeclareMathOperator{\SL}{SL}
\DeclareMathOperator{\Sp}{Sp}
\DeclareMathOperator{\Spec}{Spec}
\DeclareMathOperator{\Specb}{\mathbf{Spec}}

\let\ph\varphi

\newcommand*{\lst}[3][1]{\ensuremath{#2_{#1}, \ldots, #2_{#3}}\xspace}

\newtheorem{proposition}{Proposition}[section]

\newtheorem{lemma}[proposition]{Lemma}

\newtheorem{cor}[proposition]{Corollary}
\newtheorem{convention}[proposition]{Convention}

\theoremstyle{remark}
\newtheorem{note}[proposition]{Remark}

\theoremstyle{definition}
\newtheorem{definition}[proposition]{Definition}

\newtheorem{Not}[proposition]{Notation}

\author{Serge Lvovski}

\address{National Research University Higher School of Economics,\hfil\break
  Moscow, Russia\hfil\break
Federal Scientific Centre Science Research Institute 
of System Analysis at  Russian Academy of Science 
(FNP FSC SRISA  RAS)
} 

\email{lvovski@gmail.com}

\title{On monodromy in families of elliptic curves over \C}



\thanks{The study has been funded by the Russian Academic Excellence
  Project '5-100'.}

\begin{document}

\begin{abstract}
We show that if we are given a smooth non-isotrivial family of
elliptic curves over~\C with a smooth base~$B$ for which the general
fiber of the mapping $J\colon B\to\Aa^1$ (assigning $j$-invariant of
the fiber to a point) is connected, then the monodromy group of the
family (acting on $H^1(\cdot,\Z)$ of the fibers) coincides with
$\SL(2,\Z)$; if the general fiber has $m\ge2$ connected components,
then the monodromy group has index at most~$2m$ in $\SL(2,\Z)$. By
contrast, in \emph{any} family of hyperelliptic curves of genus
$g\ge3$, the monodromy group is strictly less than $\Sp(2g,\Z)$.

Some applications are given, including that to monodromy of hyperplane
sections of Del Pezzo surfaces.
\end{abstract}

\maketitle

\section*{Introduction}

It is believed that if fibers in a family of algebraic varieties
``vary enough'' then the monodromy group acting on the cohomology of
the fiber should be in some sense big. Quite a few results have been
obtained in this direction. See for example \cite{CojocaruHall} for
families of elliptic curves, \cite{Hall2008} for families of
hyperelliptic curves, \cite{Arias} for families of abelian varieties
(see also~\cite{Hall2011} for abelian varieties in the arithmetic
situation). In the cited papers cohomology means ``\'etale cohomology
with finite coefficients''. In this paper we address the question of
``big monodromy'' for families of elliptic curves over~\C and singular
cohomology.

The main result of the paper (Proposition~\ref{numcompts}) asserts
that {\em if $\pi\colon\X\to B$ is a smooth non-isotrivial family of
  elliptic curves over~\C and if the general fiber of its ``$J$-map''
  $J_\X\colon B\to\Aa^1$ \textup(assigning to each point of the base
  the $j$-invariant of the fiber\textup) is connected, then the
  monodromy group of the family \X is the entire group $\SL(2,\Z)$,
  and if the general fiber has $m\ge2$ connected components, then the
  monodromy group of the family \X is a subgroup of index at most $2m$
  in $\SL(2,\Z)$.} Here, by monodromy group we mean that acting on
$H^1(\cdot,\Z)$ of the fiber.

An immediate consequence of this proposition is that, \emph{in any
non-isotrivial family of elliptic curves, the monodromy group has
finite index in} $\SL(2,\Z)$ (Corollary~\ref{finite.index}). This
requires some comments.

The above assertion is similar to a well-known result about elliptic
curves over number fields, viz. to Serre's Theorem~3.2 from Chapter~IV
of~\cite{Serre:book}. It is possible that one can prove our
Corollary~\ref{finite.index} by imitating, \emph{mutatis mutandis},
Serre's proof of this theorem or even derive it from Serre's theorem
or similar arithmetical results. One merit of the approach presented
in this paper is that the proofs are very simple and elementary. One
should add that the similarity between arithmetic and geometric
situations is not absolute. For example, Theorem 5.1
from~\cite{Hall2008} could suggest that, over \C, the monodromy group
for some families of hyperelliptic curves of genus $g$ should be the
entire~$\Sp(2g,\Z)$. However, as we show in
Proposition~\ref{mono:hell:cor}, for \emph{any} family of
hyperelliptic curves of genus $g\ge3$ over~\C the monodromy
group acting on $H^1(\cdot,\Z)$ of the fiber is a proper subgroup of
$\Sp(2g,\Z)$.

Our main result has three simple consequences, which are presented in
Section~\ref{sec:appl}. First, any smooth (i.e., without degenerate
fibers) family of elliptic curves over a smooth base with commutative
fundamental group, must be isotrivial
(Proposition~\ref{isotrivial}). Second, for non-isotrivial families we
obtain an upper bound on the index of the monodromy group in
$\SL(2,\Z)$ (Proposition~\ref{lower_bound}). Third, in the case of
smooth elliptic surfaces we use Miranda's results from~\cite{Miranda}
to obtain an upper bound on the index of monodromy group in terms of
singular fibers (it turns out that only fibers of the types~$\mathrm
I_n$ and~$\mathrm I^*_N$ count); see Proposition~\ref{ell_surf}.

In Section~\ref{sec:h-ell} we prove the above mentioned result about
families of hyperelliptic curves of genus $3$ or higher.

In Section~\ref{sec3}, we derive from our main result that the
hyperplane monodromy group of a smooth Del Pezzo surface (or, for Del
Pezzos of degree $2$, the monodromy group acting on $H^1(\cdot,\Z)$ of
smooth elements of the anticanonical linear system) is the entire
$\SL(2,\Z)$ (Proposition~\ref{delPezzo}). I realize that it is not
the only way to obtain this result. Observe that, in view of
Proposition~\ref{mono:hell:cor}, Proposition~\ref{delPezzo} cannot be
extended to surfaces with hyperelliptic hyperplane sections.

Sections~\ref{sec1} through~\ref{sec:twists} are devoted to auxiliary
material (Proposition~\ref{quad_tw.=>mono} may be of some independent
interest).

\subsection*{Acknowledgements}

I am grateful to Yu.~Burman, Andrey Levin, Sergey Rybakov, Ossip
Schwarzman, and Yuri Zarhin for useful discussions.

\subsection*{Notation and conventions}

All our algebraic varieties are defined over~\C and reduced, so they
are essentially identified with their sets of closed points; the only
exception is the discussion of the notion of quadratic twist in
Section~\ref{sec:twists}. If $X$ is an algebraic variety, then
$X_{\mathrm{sm}}$ is its smooth locus and $X_{\mathrm{sing}}$ is its
singular locus.

When we say ``a general $X$ has property~$Y$'', this always means
``property $Y$ holds for a Zariski open and dense set of $X$'s''. The
word ``generic'' is used in the scheme-theoretic sense.

If $B$ is an algebraic variety and $\pi\colon \X\to B$ is a proper and
flat morphism such that a general fiber of $\pi$ is, say, a smooth
curve of genus~$1$, we will say that $\pi$ is a \emph{family of curves
  of genus~$1$}. If, in addition, the morphism $f$ is smooth, we will
say that $\pi$ (or just \X if there is no danger of confusion) is a
\emph{smooth family}, or \emph{a family of smooth varieties}.  If
$\pi\colon \X\to B$ is a family over $B$ and $f\colon B'\to B$ is a
morphism, then by $p'\colon f^*\X\to B'$ we mean the pullback
of~\X along~$f$.

By $\pi_1$ of an algebraic variety over~\C we always mean fundamental
group in the classical (complex) topology.

As usual, we put
$\Gamma(2)=\left\{A\in\SL(2,\Z)\colon A\equiv I\pmod2\right\}$,
where $I$ is the identity matrix.

Finally, we fix some terminology and notation concerning elliptic curves.

Following Miranda~\cite{Miranda}, we distinguish between curves of
genus~$1$ and elliptic curves: by elliptic curve over a field $K$ we
mean a smooth projective curve over~$K$ of genus~$1$ with a
distinguished $K$-rational point.

Similarly, by a \emph{smooth family of curves of genus $1$} we will mean a
smooth family $\pi\colon \X\to B$ such that its fibers are
curves of genus~$1$, and by a \emph{smooth family of elliptic curves} we mean
a pair $(\X,s)$, where $\X\to B$ is a smooth family of curves of genus~$1$
and $s\colon B\to\X$ is a section.

To each curve~$C$ of genus~$1$ over a filed~$K$ one can assign its
\emph{$j$-invariant} $j(C)\in K$; recall that if $C$ is (the smooth
projective model of) the curve defined by the Weierstrass equation
$y^2=x^3+px+q$, then
\begin{equation}\label{j:def}
j(C)=1728\cdot\frac{4p^3}{4p^3+27q^2}.
\end{equation}
Two curves of genus~$1$ over~\C are isomorphic if and only if their
$j$-invariants are equal. 

We say that a family over $B$ is isotrivial if it becomes trivial
after a pullback along a generically finite morphism $B_1\to B$. For
families of curves of genus~$1$ this is equivalent to the condition
that $j$-invariants of all fibers are the same.

\section{Generalities on monodromy groups}\label{sec1}

Suppose that $B$ is an irreducible variety and $\pi\colon \X\to B$ is a
family of smooth varieties.

If $b\in B_{\mathrm{sm}}$, $k\in\N$, and $G$ is an abelian group, then
the fundamental group $\pi_1(B_{\mathrm{sm}},b)$ acts on
$H^k(p^{-1}(b),G)$.

\begin{definition}
The image (corresponding to this action) of $\pi_1(B_{\mathrm{sm}},b)$
in $\Aut(H^k(p^{-1}(b),G))$ will be called \emph{monodromy group} of
the family~\X at~$b$ and denoted $\Mon(\X,b)$ (we suppress the
mention of $k$ and $G$; there will be no danger of confusion).
\end{definition}

Since $B$ is irreducible, $B_{\mathrm{sm}}$ is path connected. 
Hence, if we fix once and for all the group
$A=\Aut(H^k(p^{-1}(b_0),G))$ for some $b_0\in B_{\mathrm{sm}}$, then
all the groups $\Mon(\X,b)$ define the same conjugacy class of
subgroups of~$A$; this class (or, abusing the language, any subgroup
belonging to this class) will be denoted by $\Mon(\X)$.

In the sequel we will be working with families of smooth curves of
genus $g$ (in most cases $g$ will be equal to~$1$) as fibers and
monodromy action on $H^1$ of the fiber. Since monodromy preserves the
intersection form, the subgroups $\Mon(\X)$, where \X is such a
family, will be defined up to an inner automorphism of the
group~$\Sp(2g,\Z)$ ($\SL(2,\Z)$ if $g=1$).

\begin{convention}\label{conv:mono}
If $\pi\colon \X\to B$ is a non-smooth family, then by $\Mon(\X)$ we
mean $\Mon(\X|_U)$, where $U\subset B$ is the Zariski open subset over
which $\pi$ is smooth.  
\end{convention}

Below we list some simple properties of monodromy groups.

\begin{proposition}\label{prop:epi}
Suppose that $B$ is an irreducible variety, $U\subset B$ is a
non-empty Zariski open subset, and \X is a smooth family
over~$B$. Then $\Mon(\X|_U)=\Mon(\X)$.
\end{proposition}

\begin{proof}
The result follows from the fact that, for any $b\in U\cap
B_{\mathrm{sm}}$, the natural homomorphism $\pi_1(U\cap
B_{\mathrm{sm}},b)\to \pi_1(B_{\mathrm{sm}},b)$ is epimorphic (see for
example~\cite[0.7(B)~ff.]{FulLar}).
\end{proof}

\begin{proposition}\label{prop:pullback}
Suppose that $B'$ and $B$ are smooth irreducible varieties and \X is a smooth
family over~$B$. If $f\colon B'\to B$ is a dominant morphism such that
a general fiber of $f$ has $m$ connected components, then
$\Mon(f^*\X)$ is conjugate to a subgroup of $\Mon(\X)$, of index at
most~$m$.
\end{proposition}

\begin{cor}\label{conn.fiber:gen}
Suppose that $B'$ and $B$ are smooth irreducible varieties and \X is a smooth
family over~$B$. If $f\colon B'\to B$ is a dominant morphism such that
a general fiber of $f$ is connected, then $\Mon(f^*\X)$ is conjugate to
$\Mon(\X)$.
\end{cor}

\begin{proof}[Proof of Proposition~\ref{prop:pullback}]
It follows from~\cite[Corollary 5.1]{Verdier} and the algebraic
version of Sard's theorem that there exists a Zariski open non-empty
$U\subset B$ such that all the fibers of $f$ over points of $U$ are
smooth and the induced mapping $f'\colon f^{-1}U \to U$ is a locally
trivial bundle in the complex topology. Proposition~\ref{prop:epi}
implies that $\Mon(\X|_U)=\Mon(\X)$.  Since $f^{-1}(U)$ is (path)
connected, each fiber of this bundle has $m$ connected components, and
the base is locally path connected, $f'_*(\pi_1(f^{-1}U, x))$ is a
subgroup of index at most~$m$ in $\pi_1(U,f(x))$ for any $x\in f^{-1}(U)$.
This implies the proposition.
\end{proof}

\section{Some remarks on 3-braids}\label{sec:braids}

In this section, all topological terms will refer to the classical
(complex) topology.

We begin with some remarks on 3-braids (not claiming to novelty).

Let $\C^{(3)}$ stand for the configuration space of unordered triples
of distinct points in the complex plane. It is well known that
$\pi_1(\C^{(3)})\cong B_3$, where $B_3$ is braid group with $3$
strands. If $(u,v,w)\in\C^{(3)}$ is an unordered triple, we will write
$B_3(u,v,w)$ instead of $\pi_1(\C^{(3)},(u,v,w))$.

For any triple $(u,v,w)\in\C^{(3)}$, we denote by $X_{u,v,w}$ the
elliptic curve which is the smooth projective model of the curve with
equation $y^2=(x-u)(x-v)(x-w)$.  We are going to define a homomorphism
\[
\mu\colon B_3(u,v,w)\to \SL(H^1(X_{u,v,w},\Z))\cong\SL(2,\Z).
\]

To wit, it is well known that any braid $\gamma\in B_3(u,v,w)$ can be
represented by a homeomorphism $\ph_\gamma\colon \C\to\C$ such that
$\ph_\gamma(\{u,v,w\})=\{u,v,w\}$ and $\ph_\gamma$ is identity outside
a bounded set. Putting $\PP^1=\C\cup\{\infty\}\supset\C$, we extend
$\ph_\gamma$ to a homeomorphism from $\PP^1$ to itself by putting
$\ph_\gamma(\infty)=\infty$. If $\pi\colon X_{u,v,w}\to\PP^1\supset\C$
is the morphism induced by the projection $(x,y)\mapsto x$, then there
exists a unique homeomorphism $\tilde\ph_\gamma\colon X_{u,v,w}\to
X_{u,v,w}$ such that $\pi\circ\tilde\ph_\gamma=\ph_\gamma\circ\pi$ and
$\tilde\ph_\gamma=\mathrm{id}$ on $\pi^{-1}(\PP^1\setminus K)$, where
$K\subset\C$ is the compact set outside of which
$\pi_\gamma=\mathrm{id}$. The automorphism
\[
\tilde\ph_\gamma^*\colon H^1(X_{u,v,w},\Z)\to H^1(X_{u,v,w},\Z)
\]
does not depend on the choice of the $\phi_\gamma$ representing
$\gamma$, and we put $\mu(\gamma)=\tilde\ph_\gamma^*$.

\begin{proposition}\label{(AB)^3}
If $\Gamma\in B_3(u,v,w)$ is the braid represented by the
loop in $\C^{(3)}$ defined by the
 formula $t\mapsto (ue^{2\pi it}, ve^{2\pi it}, ve^{2\pi it})$, $t\in
[0;1]$, then $\mu(\Gamma)=\left(\begin{smallmatrix}-1&0\\ 0&-1\end{smallmatrix}
\right)$.
\end{proposition}

\begin{proof}
To prove the proposition, we choose
generators of $B_3(u,v,w)$ and a basis in $H_1(X_{u,v,w},\Z)$.

To fix generators of the braid group, we choose the points $u, v,
w\in\C$ so that they are collinear and $v$ lies between $u$ and
$w$. Now let $A$ and~$B$ be the braids corresponding to the following
closed paths in $\C^{(3)}$: in the path defining $A$, the point $w$
stays where it is while $u$ and $v$ are swapped, $u$ and $v$ moving
along small arcs close to the segment $[p,q]$ so that the composition
of paths traveled by $u$ and $v$ defines a positively oriented simple
closed curve. The braid $B$ is defined similarly, with the point $u$
staying put and the points $v$ and $w$ being exchanged; see
Figure~\ref{fig:A,B}.
\begin{figure}
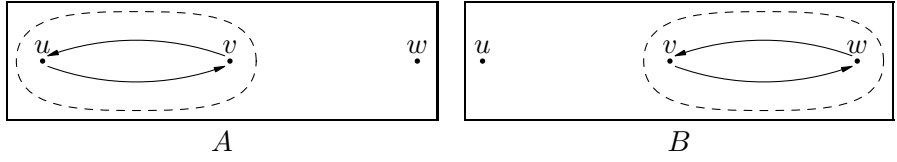

  \centering
\begin{tabular}{cc}
\fbox{\includegraphics{pics1.mps}}&\fbox{\includegraphics{pics2.mps}}\\[2pt]
$A$&$B$
\end{tabular}
  \caption{Two generators of $B_3(u,v,w)$. One may assume that
    homeomorphisms of \C representing these braids are identity
    outside the dashed ovals.}\label{fig:A,B}
  
\end{figure}
The group $B_3(u,v,w)$ is generated by $A$ and
$B$, and these braids satisfy the relation $ABA=BAB$.

For a basis in
$H_1(X_{u,v,w},\Z)$ we choose the $1$-cycles $\alpha$ and $\beta$
that are obtained by lifting the closed paths
$\alpha$ and $\beta$ on Fig.~\ref{fig:alpha,beta} from \C to
$X_{u,v,w}$.

\begin{figure}
\centering  
\includegraphics{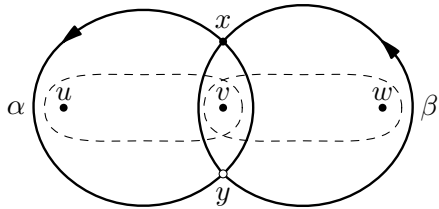}
  \caption{Projections to~\C of the cycles $\alpha$ and $\beta$. Point
    $x$ corresponds to their intersection point on $X_{u,v,w}$,
    point~$y$ is just an apparent intersection point.}\label{fig:alpha,beta}
\end{figure}

Abusing the language, we will denote the action on $H_1(X_{u,v,w},\Z)$
of a homeomorphism representing the braid~$A$ by the same letter~$A$,
and similarly for~$B$. 
Since the homeomorphisms
representing $A$ and $B$ can be chosen to be identity outside the
corresponding dashed ovals on Fig.~\ref{fig:A,B}, it is clear
that $A(\alpha)=\alpha$ and
$B(\beta)=\beta$. Taking into account that $A$ and $B$ preserve the
intersection pairing on $H_1(X_{u,v,w},\Z)$, one concludes that, in the basis
$(\alpha,\beta)$, the action of $A$ and $B$ on $H_1$ is given by
matrices of the form
\[
A=
\begin{pmatrix}
  1&x\\0&1
\end{pmatrix},\quad
B=
\begin{pmatrix}
  1&0\\
y&1
\end{pmatrix},\qquad x,y\in\Z.
\]
The relation $ABA=BAB$ implies that
\begin{equation}\label{eq:ABA=BAB}
x=2x+x^2y,\quad y=2y+xy^2.
\end{equation}
Equations~\eqref{eq:ABA=BAB} imply that either $x=y=0$ or $xy=-1$. The
first case is impossible: if both $A$ and $B$ act as identity, then
the entire braid group acts identically, which is absurd since its
action is non-trivial (see for example~\cite{Acampo}). So, one of the
integers $x$ and $y$ is equal to~$1$ and the other is equal
to~$-1$. Dualizing, we see that either $\mu(A)$ acts on $H^1(X_{u,v,w},\Z)$
as $\left(\begin{smallmatrix}1&0\\-1&1\end{smallmatrix}\right)$ (in
  the basis dual to $(\alpha,\beta)$) and $\mu(B)$ acts as
 $\left(\begin{smallmatrix}1&1\\0&1\end{smallmatrix}\right)$, or vice
    versa. 

Now it is well known that $\Gamma=(AB)^3$.
Plugging the possible
values of $A$ and~$B$, one obtains the result.
\end{proof}

\begin{note}
One can show that, with the choice of signs as on
figures~\ref{fig:A,B} and~\ref{fig:alpha,beta}, one has
\[
\mu(A)=
\begin{pmatrix}
  1&0\\
-1&1
\end{pmatrix},\quad
\mu(B)=
\begin{pmatrix}
  1&1\\0&1
\end{pmatrix}.
\]
We do not need to be that precise.
\end{note}

\section{Quadratic twists and monodromy}\label{sec:twists}

In this and the following section we will be studying monodromy groups
acting on $H^1(\cdot,\Z)$ of fibers in families of smooth curves of
genus~$1$. In
such families, the monodromy group acting on $H^1$ of the fiber is
contained in $\SL(2,\Z)$.

If $\X\to B$ is a smooth family of elliptic curves, then the morphism
$B\to\Aa^1$ assigning the $j$-invariant $j(p^{-1}(b))$ to a point
$b\in B$, will be denoted by $J_\X$. Following
Miranda~\cite[Lecture~V]{Miranda}, we will say that $J_\X$ is
\emph{the $J$-map} of the family~\X (in Kodaira's
paper~\cite{Kodaira}, the morphism $J_\X$ is called \emph{analytic
  invariant of the family~\X}).

\begin{Not}
If \X is a smooth family of elliptic curves over $B$ and if $b_0\in
B$, then the monodromy representation $\pi_1(B,b_0)\to
\SL(H^1(\X_{b_0},\Z))$ will be denoted by $\rho_\chi$.
\end{Not}

Suppose now that $p\colon X\to B$ is a family of elliptic curves over
a smooth and connected base. Since its fiber over the generic point of
$B$ is an elliptic curve over the field of rational functions $\C(B)$,
and since this elliptic curve can be reduced to the Weierstrass normal
from, there exists a Zariski open subset $U\subset B$ such that the
restriction $\X|_U$ is isomorphic to the family
\begin{equation}\label{eq:Weierstrass}
w^2=z^3+Pz+Q,
\end{equation}
where $P$ and $Q$ are regular functions on $U$, the fiber over $b\in
B$ being the smooth projective model of the curve defined by the equation
$w^2=z^3+P(b)z+Q(b)$, and discriminant of the right-hand side
of~\eqref {eq:Weierstrass} does not vanish on~$U$.
Proposition~\ref{prop:epi} shows that $\Mon(\X)=\Mon(\X|_U)$, so, as
far as monodromy groups are concerned, we may and will assume that
$U=B$ and that the family is defined by~\eqref{eq:Weierstrass} with
non-vanishing discriminant.

Any such family of the form~\eqref{eq:Weierstrass} defines a morphism
$\Br_\X\colon B\to \C^{(3)}$ assigning to each point $b\in B$ the
collection of roots of $z^3+P(b)z+Q(b)$. If $b_0\in B$ and if
$\{u,v,w\}$ is the set of roots of the polynomial $z^3+P(b)z+Q(b)$,
then the
morphism $\Br_\X$ induces a homomorphism $\br_\X\colon \pi_1(B,b_0)\to
B_3(u,v,w)$. If $\X_{b_0}$ is the fiber of \X over $b_0$, and if
\[
\mu\colon B_3(u,v,w)\to \SL(H^1(X_{b_0}),\Z))=\SL(2,\Z)
\]
is the homomorphism defined in Section~\ref{sec:braids}, then the diagram
\[
\xymatrix{
{\pi_1(B,b_0)}\ar[rr]^{\rho_\chi}\ar[dr]_{\br_\chi}&&{\SL(H^1(X_{b_0}),\Z))}\\
&{B_3}\ar[ur]_\mu
}
\]
is commutative.

Suppose that $p_1\colon\X_1\to B$ and $p_2\colon \X_2\to B$ are two
families of elliptic curves over a base~$B$. One says that $\X_1$
and $\X_2$ \emph{differ by a quadratic twist} if their
scheme-theoretic generic fibers (which are elliptic curves over the
field of rational functions $K=\C(B)$) are isomorphic over a quadratic
extension of~$K$. It it clear that this is the case if and only if
there exists a morphism of degree~$2$ (not necessarily finite or
\'etale) $f\colon U\to B$ such that $f^*\X_1$ and $f^*\X_2$ are
isomorphic smooth families.
If the families $\X_1$ and $\X_2$ differ by a quadratic twist, then
they can be represented by Weierstrass equations
\begin{equation}\label{eq:qu_tw}
\begin{aligned}
\X_1:\quad y^2&=x^3+Px+Q,\\
\X_2:\quad y^2&=x^3+D^2\cdot Px+D^3\cdot Q,  
\end{aligned}
\end{equation}
where $D$ is a rational function on $B$ (see \cite[Chapter~X,
  Proposition 5.4]{Silverman}).

Being interested only in the
monodromy groups $\Mon(\X_1)$ and $\Mon(\X_2)$, we can, replacing
$B$ by a Zariski open subset if necessary, assume that the families
$\X_1$ and $\X_2$ are smooth; in particular, this implies that $D$ is a
regular function on $B$ without zeroes. 

Suppose that $B$ is a smooth algebraic variety, $D$ is a regular
function on $B$ without zeroes, and $b_0\in B$ is a point. In the
definition that follows we regard $B$ as a complex manifold and $D$ as
a holomorphic function on~$B$.

\begin{definition}
In the above setting, by $\chi_D\colon\pi_1(B,b)\to\{\pm1\}$ we denote
the homomorphism defined as follows. If $B_0\in B$,
$\gamma\in\pi_1(B,b_0)$, we put $\chi_D(\gamma)=-1$ if the function
$\sqrt D$ changes after the analytic continuation along a loop
representing $\gamma$, and we put $\chi_D(\gamma)=1$ otherwise. In
other words, if a loop representing $\gamma$ is of the form $t\mapsto
\ph(t)$, $t\in[0;1]$, then $\chi_D(\gamma)=(-1)^k$, where $k$ is the
number of times the loop $D\circ\ph$ winds around the origin.

We will say that $\chi_D$ is the \emph{quadratic character associated
  to $D$}.
\end{definition}

\begin{proposition}\label{quad_tw.=>mono}
In the above setting, suppose that $\X_1$ and $\X_2$ are smooth
families of elliptic curves that differ by a quadratic twist as
in~\eqref{eq:qu_tw}. Then the monodromy homomorphism
$\rho_{\chi_2}\colon \pi_1(B)\to\SL(2,\Z)$ differs from
$\chi_D\rho_{\chi_1}$ by an inner automorphism of $\SL(2,\Z)$.
\end{proposition}

\begin{proof}
Suppose that $\X_1$ and $\X_2$ are defined by the
equations~\eqref{eq:qu_tw}, where $D$ has no zeroes or poles on $B$
and discriminants of the left-hand sides of never vanish. If $u(b)$,
$v(b)$, and $w(b)$ are the roots of the polynomial $x^3+P(b)x+Q(b)$,
where $b\in B$, then roots of the polynomial
$x^3+D(b)^2P(b)x+D(b)^3Q(b)$ are $D(b)u(b)$, $D(b)v(b)$, and
$D(b)w(b)$.

In the argument that follows we will not distinguish between path and
loops in $\C^{(3)}$ and their homotopy classes; this will not lead to
a confusion. That said, choose a base point $b_0\in B$ and fix a path
$\tau$ in $\C^{(3)}$ joining the points (unordered triples)
$(D(b_0)u(b_0), D(b_0)v(b_0),D(b_0)w(b_0))$ and $(u(b_0),
v(b_0),w(b_0))$. If $\gamma\in\pi_1(B,b_0)$, then
\[
\br_{\X_2}(\gamma)=\tau\br_{\X_1}(\gamma)\delta\tau^{-1},
\]
where $\delta\in
B_3 (u(b_0), v(b_0),w(b_0))$ is the loop defined by the formula
\[
t\mapsto (D(\gamma(t))/|D(\gamma(t)|)\cdot \gamma(t)
\]
in which we use the following notation: if $\lambda\in\C^*$ and
$\gamma=(u,v,w)\in\C^{(3)}$, then $\lambda\cdot\gamma$ is the
unordered triple $(\lambda u, \lambda v, \lambda w)$.

If the loop $\delta$ winds $k$ times around the origin, then
Proposition~\ref{(AB)^3} implies that
$\mu(\delta)=(-1)^kI=\chi_D(\gamma)I$, where $I$ is the identity
matrix, whence the result.
\end{proof}

\begin{lemma}\label{abstr.tw.}
Suppose that $\Pi$ is a group and that $\rho\colon\Pi\to\SL(2,\Z)$ and
$\chi\colon \Pi\to\{\pm1\}$ are homomorphisms. Put $G_1=\im\rho$,
$G_2=\im(\chi\rho)$. Then one of the following cases holds:

\textup{(i)} $G_1=G_2$;

\textup{(ii)} 
there exists a subgroup $H\subset G_1$, $(G_1:H)=2$, such that
$G_2=H\cup (-I)(G_1\setminus H)$;

\textup{(iii)}
$G_2$ is the subgroup of $\SL(2,\Z)$ generated by $G_1$ and~$-I$.
\end{lemma}

\begin{proof}
If the quadratic character $\chi$ is trivial, then $G_2=G_1$ and
case~(i) holds. Suppose now that $\chi$ is non-trivial; then
$\Ker\chi$ is a subgroup of index $2$ in $\Pi$.

If $\Ker\rho\subset\Ker\chi$, then the character $\chi$
factors through the group $G_1$:
\[
\xymatrix{
\Pi\ar@{->>}[r]^{\rho}\ar[dr]_{\chi}&G_1\ar[d]^{\bar\chi}\\
  &{\{\pm1\};}
}
\]
now if $H=\Ker\bar\chi$, then $G_2=H\cup (-I)(G_1\setminus H)$. 

Suppose finally that $\Ker\rho\not\subset\Ker\chi$. Then
$G_1=\rho(\Ker\chi)=\rho(\Pi\setminus\Ker\chi)$ and
\[
G_2=\rho(\Ker\chi)\cup\chi\rho(\Pi\setminus\Ker\chi)
=\rho(\Ker\chi)\cup(-I)(\rho(\Pi\setminus\Ker\chi))=G_1\cup(-I)G_1,
\]
so $G_2$ is the subgroup generated by $G_1$ and $-I$ and case~(iii) holds.
\end{proof}

Suppose now that $\X_1$ and $\X_2$ are smooth families of elliptic
curves over the same base $B$. If we fix a base point $b\in B$,  we can identify (not canonically) first integer cohomology groups of the
fibers $(\X_1)_b$ and $(\X_2)_b$ and identify them both with $\Z^2$.

\begin{proposition}\label{modify_mono}
Suppose that $\X_1$ and $\X_2$ are smooth families of elliptic curves
over the same base $B$ and that $\X_1$ and $\X_2$ differ by a
quadratic twist with a non-vanishing regular function~$D$ as
in~\eqref{eq:qu_tw}.  Put $G_i=\Mon(\X_i, b)\subset\SL(2,\Z)$
\textup(these subgroups are only defined up to a
conjugation\textup). Then:

\textup{(i)}
either $G_1$ and $G_2$ are conjugate, or one of these groups contains a
subgroup of index $2$ that is conjugate to the other subgroup, or
each $G_i$ contains a subgroup $H_i$ of index~$2$ and the subgroups
$H_1$ and $H_2$ are conjugate.

\textup{(ii)} 
if $G_1=\SL(2,\Z)$, then $G_2=\SL(2,\Z)$.
\end{proposition}

\begin{proof}
 Put $\Pi=\pi_1(B)$.  Proposition~\ref{quad_tw.=>mono} implies that,
 conjugating the subgroups if necessary, one may assume that there
 exist homomorphisms $\rho\colon \Pi\to\SL(2,\Z)$ and $\chi\colon
 \Pi\to\{\pm1\}$ such that $G_1=\im(\rho)$, $G_2=\im(\chi\rho)$.

We prove part~\textup{(i)} first; to that end, we invoke
Lemma~\ref{abstr.tw.}. If case~\textup{(i)} or case~\textup{(iii)} of
this lemma holds, we are done. In case~(ii) there exists a subgroup
$H\subset G_1$, $(G_1:H)=2$, such that $G_2=H\cup (-I)(G_1\setminus
H)$.

If $-I\in H$, this implies that $G_2=G_1$; if $-I\in G_1\setminus H$,
this implies that $G_2=H$, so $G_2$ is a subgroup of index $2$ in
$G_1$; in the remaining case $-I\notin G_1$, one has $G_2\ne G_1$,
$G_2\supset H$ and $(G_2:H)=(G_1:H)=2$. Thus, part~\textup{(i)} is proved. 

To prove part~\textup{(ii)}, suppose that $G_1=\SL(2,\Z)$. If
case~\textup{(i)} or case~\textup{(iii)} of Lemma~\ref{abstr.tw.}
holds, it is clear that $G_2=\SL(2,\Z)$. If case~\textup{(ii)}, observe
that $\SL(2,\Z)$ contains a unique subgroup $H$ of index~$2$:
this follows from the fact that the abelianization of $\SL(2,\Z)$ is
$\Z/12\Z$. Expressing the corresponding epimorphism $\ph\colon
\SL(2,\Z) \to\Z/12\Z$ in terms of its action on the generators
$\left(\begin{smallmatrix} 0&1\\-1&0
\end{smallmatrix}
\right)$ and $\left(
\begin{smallmatrix}
  0&-1\\1&1
\end{smallmatrix}\right)$, 
one sees that $\ph(-I)\in 6\Z/12\Z$, whence $-I\in H$, whence
$G_2=G_1=\SL(2,\Z)$.
\end{proof}

\begin{cor}[from the proof]\label{twice_greater}
Suppose that $\X_1$ and $\X_2$ are families of elliptic curves over
the same smooth base that differ by a quadratic twist. Then

\textup{(i)} if the monodromy group $\Mon(\X_1)$
has finite index in $\SL(2,\Z)$, then either the indices
$(\SL(2,\Z):\Mon(\X_1))$ and $(\SL(2,\Z):\Mon(\X_2))$ are equal or
one of them is twice greater than the other;

\textup{(ii)} images of $\Mon(\X_1)$ and $\Mon(\X_2)$ in $\PSL(2,\Z)$ are
conjugate.\qed
\end{cor}

\begin{proposition}\label{twist=>mono}
Suppose that \X and \Y are smooth families of elliptic curves over the
same base $B$ and that their $J$-maps $J_\X,J_\Y\colon B\to \Aa^1$ are
equal and non-constant. Then

\textup{(i)} either $(\SL(2,\Z):\Mon(\X))=(\SL(2,\Z):\Mon(\Y))$ or one
of these indices is twice greater than the other;

\textup{(ii)} images of $\Mon(\X)$ and $\Mon(\Y)$ in $\PSL(2,\Z)$ are
conjugate;

\textup{(iii)} if $\Mon(\X)=\SL(2,\Z)$ then $\Mon(\Y)=\SL(2,\Z)$.
\end{proposition}

\begin{Not}
In the next section we will see that if the $J$-map is not constant
then the monodromy group has finite index in $\SL(2,\Z)$. In the
statement of this proposition we allow indices of subgroups to be
infinite and assume that $2\cdot\infty=\infty$.
\end{Not}

\begin{proof}
In view of Propositions~\ref{twice_greater} and~\ref{modify_mono} it
suffices to show that \X and \Y differ by a quadratic twist. To that
end put $K=\C(B)$ (the field of rational functions).  The
(scheme-theoretic) generic fibers of the families \X and \Y over
$\Spec K$ are elliptic curves over $K$. They have the same
$j$-invariant $J_\X=J_\Y\in \C(B)$, and this $j$-invariant is not
equal to $0$ or~$1728$ since $J_\X=J_\Y$ is not constant. Hence, these
elliptic curves differ by a quadratic twist by virtue of Proposition
5.4 from \cite[Chapter~X]{Silverman}, and so are the corresponding
families.
\end{proof}

\section{Main result and applications}\label{sec:appl}

We begin with a folklore result for which I do not know an adequate reference.

\begin{proposition}\label{folklore}
Suppose that $\pi\colon\X\to B$ is a smooth family of curves of
genus~$1$, where $B$ is a variety \textup(i.e., a reduced scheme of
finite type over~\C{}\textup). Then the mapping from $B$ to \C that
assigns $j$-invariant $j(f^{-1})(b)$ to a point $b\in B$, is induced
by a morphism from $B$ to $\Aa^1$.
\end{proposition}

\begin{proof}
If $\X'=\Pic^0(\X/B)$ (relative Picard variety, see~\cite[Section
  5]{Kleiman}), then the family $p'\colon \X'\to\C$ has a section (to
wit, the zero section) and induces the same mapping from $B$ to \C
since $\Pic^0(C)\cong C$ if $C$ is a smooth curve of genus~$1$
over~\C. Thus, without loss of generality one may assume that the
family in question has a section; in this case
see~\cite[\S\,5]{formulaire}.
\end{proof}

\begin{proposition}\label{numcompts}
Suppose that $\pi\colon\X\to B$ is a smooth family of curves of
genus~$1$ over a smooth and connected base~$B$ \textup(the ground field
is~\C{}\textup); let $J_\X\colon B\to\Aa^1$ be the $J$-map, attaching to any
point $a\in B$ the $j$-invariant of the fiber of \X over~$a$.  

\textup{(i)} If the morphism~$J_\X$ is not constant and its general
fiber is connected, then $\Mon(X)=\SL(2,\Z)$.

\textup{(ii)} If the morphism~$J_\X$ is not constant and its general
fiber has $m\ge$ connected components, then $\Mon(\X)$ is a subgroup of
index at most~$2m$ in $\SL(2,\Z)$ and the image of $\Mon(\X)$ in
$\PSL(2,\Z)$ is a subgroup of index at most $m$ in~$\PSL(2,\Z)$.
\end{proposition}

We begin with a lemma.

\begin{lemma}\label{Pic}
Suppose that $\pi\colon \X\to B$ is a smooth family of curves of
genus~$1$. Then there exists a smooth family of elliptic curves
$p'\colon \X'\to B$ such that the $J$-maps $J_{\X},J_{\X'}\colon
B\to\Aa^1$ are the same and $\Mon(\X')\subset\SL(2,\Z)$ is conjugate
to $\tau(\Mon(\X))$, where $\tau\colon\SL(2,\Z)\to\SL(2,\Z)$ is the
automorphism defined by the formula $M\mapsto(M^t)^{-1}$.
\end{lemma}

\begin{proof}
Put $\X'=\Pic^0(\X/B)$.  As we have seen in the proof of
Proposition~\ref{folklore}, $J_\X=J_{\X'}$ and the family $p'\colon
\X'\to B$ has a section. Finally, $R^1p'_*\Z\cong \HOM(R^1p_*\Z,\Z)$,
where by \Z we mean the constant sheaf with the stalk~\Z (see for
example~\cite[\S\,9]{Mumford}), and this implies the assertion about
monodromy.
\end{proof}

\begin{proof}[Proof of Proposition~\ref{numcompts}]
Lemma~\ref{Pic} implies that we may assume that the family $\pi\colon
\X\to B$ in question has a section. Assuming that, put
\[
V=\{(p,q)\in\Aa^2\colon 4p^3+27q^2\ne0\}
\]
and consider the smooth family of
elliptic curves $\Bc\to V$ in which the fiber over $(p,q)$ is the
smooth projective model~$C_{p,q}$ of the curve with equation
$y^2=x^3+px+q$ and the section assigns to $(p,q)$ the ``point at
infinity'' of this model. It is well known (see for example \cite[Corollary to
Theorem~1]{Acampo}) that $\Mon(\Bc)=\SL(2,\Z)$.

Now put $\Aa^1_0=\Aa^1\setminus\{0\}$,
$\Aa^1_{0,1728}=\Aa^1\setminus\{0,1728\}$ and
\[
V'=J_\Bc^{-1}(\Aa^1_{0, 1728})=\{(p,q)\in\Aa^2\colon
4p^3+27q^2\ne0,\  p\ne0,q\ne0\}. 
\]
Let $\Bc'$ be the restriction of the family~$\Bc$ to $V'$;
put $B'=J_\X^{-1}(\Aa^1_{0,1728})$, and let $\X'$ be the restriction of \X
to $B'$. Proposition~\ref{prop:epi} implies that
$\Mon(\X')=\Mon(\X)$ and $\Mon(\Bc')=\Mon(\Bc)=\SL(2,\Z)$.

Observe that there exists an isomorphism $g\colon V\to
\Aa^1_0\times \Aa^1_{0,1728}$ such that the diagram
\[
\xymatrix{
{V'}\ar[rr]^(.4)g\ar[dr]_{J_{\X'}}&&{\Aa^1_0\times \Aa^1_{0,1728}}\ar[dl]^{\pr_2}\\
&{\Aa^1_{0,1728}}
}
\]
is commutative. Indeed, one can define $g$ by the formula
$(p,q)\mapsto (q/p, j(C_{p,q}))$, and the inverse morphism will be
\[
(\lambda,j)\mapsto
\left(\frac{\lambda^2}{\frac4{27}\left(\frac{1728}j-1\right)},
\frac{\lambda^3}{\frac4{27}\left(\frac{1728}j-1\right)}\right).
\]
Hence, in the fibered product
\[
\xymatrix{
W\ar[r]^f\ar[d]_u&{V'}\ar[d]^{J_{\Bc'}}\\
B'\ar[r]_{J_{\X'}}&{\Aa^1_{0,1728}}
}
\]
(we mean fibered product in the category of reduced algebraic
variates, so $W$ is the scheme theoretic fibered product modulo
nilpotents) the variety $W$ is isomorphic to $(\Aa^1_0)\times B'$ (in
particular, $W$ is smooth and irreducible) and fibers of $f$ are
isomorphic to fibers of $J_{\X_0}$. Thus, the hypothesis implies that a
general fiber of the morphism $J_{\X_0}$ has $m$ connected
components. On the other hand, any fiber of the morphism $u$ is
irreducible since it is isomorphic to $\Aa^1_0$. Now
Proposition~\ref{prop:pullback} and Corollary~\ref{conn.fiber:gen}
imply that for the pullback families $f^*\Bc'$ and $u^*\X'$ on $W$,
the group $\Mon(f^*\Bc')$ is a subgroup of index at most~$m$ in
$\Mon(\Bc)=\SL(2,\Z)$ and $\Mon(u^*\X_0)=\Mon(\X_0)$ (as usual, this
equation holds up to a conjugation). 

Since $J_{f^*\Bc'}=J_{\Bc'}\circ f=J_{\X'}\circ
u=J_{u^*\X'}$, Proposition~\ref{twist=>mono} implies the result.
\end{proof}

\begin{note}
I do not know whether the bound in this proposition can be improved to
$(\PSL(2,\Z):\Mon(\X))\le m$ for $m>1$.
\end{note}

Proposition~\ref{numcompts} implies the following fact.

\begin{cor}\label{finite.index}
If $\X\to B$ is a non-isotrivial smooth family of curves of genus~$1$,
then its monodromy group is a subgroup of finite index in~$\SL(2,\Z)$.
\end{cor}

Here is the first application of what we proved.

\begin{proposition}\label{isotrivial}
If $B$ is a smooth algebraic variety with abelian fundamental group,
then any smooth family of curves of genus~$1$ over $X$ must be
isotrivial.
\end{proposition}

\begin{proof}
Suppose that $\pi\colon \X\to B$ is a smooth family of curves of
genus~$1$, where $B$ is smooth and irreducible and $\pi_1(B)=G$ is
abelian. 

We are to show that the $J$-map $J_\X\colon B\to\Aa^1$ is constant. If
this is not the case, then Corollary~\ref{finite.index} asserts that
the monodromy group $\Mon(\X)$ of the family~\X has finite index in
$\SL(2,\Z)$. Since $\Gamma(2)$ has finite index in $\SL(2,\Z)$, one
has $(\Gamma(2):\Gamma(2)\cap\Mon(\X))<\infty$. If $G$ is the image of
$\Gamma(2)\cap\Mon(\X)$ in $\Gamma(2)/\{\pm I\}$, then $G$ is
an abelian subgroup of finite index in
$\Gamma(2)/\{\pm I\}$. The latter group is isomorphic to the
free group~$F_2$ with two generators, and Schreier's theorem on
subgroups of free groups implies that $F_2$ contains no abelian
subgroup of finite index. We arrived at the desired contradiction.
\end{proof}

For the case of non-commutative $\pi_1$ of the base, one can obtain
an upper bound on the index of monodromy groups in non-isotrivial
families.

\begin{proposition}\label{lower_bound}
Suppose that $\X\to B$ is a smooth non-isotrivial family of curves of
genus~$1$ over a smooth base $B$ and that $\pi_1(B)$ can be generated by
$r\ge2$ elements. Then $(\SL(2,\Z):\Mon(\X))\leq 12(r-1)$.
\end{proposition}

\begin{cor}
Suppose that $\X\to C$ is a non-isotrivial family of elliptic curves
over a smooth curve of genus~$g$, with $s$ degenerate fibers. Then
$(\SL(2,\Z):\Mon(\X))\leq 12(2g+s-1)$. 
\end{cor}

Proposition~\ref{lower_bound} is a consequence of the following
elementary lemma.

\begin{lemma}\label{lemma:Schreier}
Suppose that $G\subset \SL(2,\Z)$ is a subgroup of finite index and
that $G$ can be generated by $r$ elements. Then $(\SL(2,\Z):G)\le
12(r-1)$.   
\end{lemma}

\begin{proof}[Proof of the lemma]
Throughout the proof, free group with $m$ generators will be denoted
by $F_m$.

Since $G$ can be generated by $r$ elements, there exists an
epimorphism $\pi\colon F_r\twoheadrightarrow G$. Putting
$H=p^{-1}(G\cap\Gamma(2))$, one obtains the following commutative diagram
of embeddings and surjections:
\begin{equation}\label{eq:Schreier}
\xymatrix{
H\ar@{->>}[r]_-{p'}\ar@{^{(}->}[d]^{\text{index}=d}&{G\cap\Gamma(2)}
\ar@{^{(}->}[d]^{\text{index}=d}\ar@{^{(}->}[r]&  
  {\Gamma(2)}\ar@{^{(}->}[d]^{\text{index}=6}\\
F_r\ar@{->>}[r]_p&G\ar@{^{(}->}[r]_-i&{\SL(2,\Z)}\ar@{->>}[r]_-j&
\rlap{$\SL(2,\Z)/\Gamma(2)\cong S_3$.}
}
\end{equation}
If $d\le 6$ is the order of $\im(j\circ i)$, then $(G:G\cap
\Gamma(2))=(F_r:H)=d$, so by Schreier's theorem $H\cong F_{d(r-1)+1}$.
Since the morphism~$p'$ is surjective, the group $G\cap \Gamma(2)$ can
be generated by $d(r-1)+1$ elements. Put $F=\Gamma(2)/\{\pm I\}$, and
let $\pi\colon \Gamma(2)\to F$ be the natural projection. The subgroup
$\pi(G\cap \Gamma(2))\subset F$ can be also generated by $d(r-1)+1$
elements; since $F\cong F_2$, Schreier's theorem implies that
$\pi(G\cap \Gamma(2))\cong F_m$, where $m\le d(r-1)+1$ (indeed, the
free group $F_n$ cannot be generated by $k<n$ elements). Applying
Schreier's for the third time, we obtain that
$(F:\pi(H\cap\Gamma(2)))=m-1\le d(r-1)$, whence $(\Gamma(2):G\cap
\Gamma(2))\le 2(m-1)\le 2d(r-1)$. It follows from the right-hand
square of the diagram~\eqref{eq:Schreier} that
\[
(\SL(2,\Z):G)=\frac{(\SL(2,\Z):\Gamma(2))\cdot(\Gamma(2):G\cap\Gamma(2))}
{(G:G\cap\Gamma(2))}\le \frac{12d(r-1)}d,
\]
whence the result.
\end{proof}

\begin{proof}[Proof of Proposition~\ref{lower_bound}]
Put $\Mon(\X)=G\subset\SL(2,\Z)$. Since $\pi_1(B)$ can be generated by
$r$ elements, the same is true for $G$; now
Corollary~\ref{finite.index} implies that $(\SL(2,\Z):G)<+\infty$, and
Lemma~\ref{lemma:Schreier} applies.
\end{proof}

Using Proposition~\ref{numcompts} one can obtain other lower bounds
for monodromy groups. Observe first that the named proposition
immediately implies the following corollary, in the statement of which
we use Convention~\ref{conv:mono}.

\begin{cor}\label{simple_cor}
If $\pi\colon\X\to C$ is a family of elliptic curves over a smooth
projective curve~$C$ and if~$J_\X\colon C\dasharrow\Aa^1$ is its
$J$-map, and if $J_\X$ is not constant, then $(\SL(2,\Z):\Mon(\X))\le
2\deg J_\X$.
\end{cor}

If \X is smooth and $\pi$ has a section, one can be more specific.

\begin{proposition}\label{ell_surf}
Suppose that $\pi\colon\X\to C$ is a minimal smooth elliptic surface
with section \textup(it means that \X is a smooth projective surface,
$C$ is a smooth projective curve, the general fiber of $\pi$ is a
smooth curve of genus~$1$, no fiber of $\pi$ contains a rational
$(-1)$-curve, and $p$ has a section\textup) and that $J_\X$ is not
constant.

Then
\begin{equation}\label{eq:ell_surf}
(\SL(2,\Z):\Mon(\X))\le 2\cdot \sum_{s\in C} e(s),
\end{equation}
where $e(s)=n$ if the fiber over $s$ is a cycle of $n$ smooth rational
curves or the nodal rational curve if $n=1$ \textup(type~$\mathrm I_n$
in Kodaira's classification~\cite{Kodaira,Miranda}\textup), $e(s)=n$
if the fiber over $s$ consists of $n+5$ smooth rational curves with
intersection graph isomorphic to the extended Dynkin graph~$\tilde
D_{n+4}$, $n\ge 1$ \textup(type $\mathrm I^*_n$ in Kodaira's
classification\textup), and $e(s)=0$ otherwise.
\end{proposition}

\begin{proof}
In view of Corollary~\ref{simple_cor} the index in the left-hand side
of~\eqref{eq:ell_surf} is less or equal to $2\deg J_\X$, and $\deg
J_\X$ equals $\sum_{s\in C} e(s)$ by virtue of Corollary IV.4.2
from~\cite{Miranda}.
\end{proof}

Similarly, one can express $\deg J_\X$ (and obtain a lower bound for
$\Mon(\X)$) using the information about the points where $j$-invariant
of the fiber (smooth or not) equals~$0$ or~$1728$, see for
example~\cite[Lemma~IV.4.5, Table~IV.3.1]{Miranda} and Table. In the notation
if~\cite{Miranda}, $j$-invariant is 1728 times less than that defined
by~\eqref{j:def}; of course, this does not affect multiplicities of
poles.

\section{A remark on families of hyperelliptic curves}\label{sec:h-ell}

\begin{proposition}\label{mono:hyper}
If $\pi\colon\X\to B$ is a smooth family of hyperelliptic curves of
genus~$g>2$, then
\[
(\Sp(2g,\Z):\Mon(\X))\ge
\frac{2^{g^2}(2^{2g}-1)(2^{2(g-1)}-1)\cdot\dots\cdot (2^2-1)}{(2g+2)!}.
\]
\end{proposition}

\begin{cor}\label{mono:hell:cor}
If $\pi\colon\X\to B$ is a smooth family of hyperelliptic curves of
genus~$g>2$, then $\Mon(\X)$ is a proper subgroup of $\Sp(2g,\Z)$.
\end{cor}

\begin{proof}[Proof of Proposition~\ref{mono:hyper}]
In this proof, $\Mon(\X,\Z)$ will denote the monodromy group acting on
the integer $H^1$ of a fiber of~\X, and $\Mon(\X,\Z/2\Z)$ will stand
for the monodromy group acting on cohomology with coefficients in
$\Z/2\Z$.

Since the reduction modulo~$2$ mapping
$\Sp(2g,\Z)\to\Sp(2g,\Z/2\Z)$ is surjective, one has
\[
(\Sp(2g,\Z):\Mon(\X,\Z))\ge(\Sp(2g,\Z/2\Z):\Mon(\X,\Z/2\Z)),
\]
so it suffices to show that
\begin{equation}\label{eq:h-ell2}
(\Sp(2g,\Z/2\Z):\Mon(\X,\Z/2\Z))\ge
\frac{2^{g^2}(2^{2g}-1)(2^{2(g-1)}-1)\cdot\dots\cdot (2^2-1)}{(2g+2)!}.
\end{equation}
To that end, let $X$ be a hyperelliptic curve of genus $g\ge2$ that is
a fiber of~\X; denote its Weierstrass points by \lst P{2g+2}. It is
well known (see for example~\cite[Lemma 2.1]{Cornelissen}) that the
$2$-torsion subgroup $(\Pic(X))_2\subset\Pic(X)$ is generated by
classes of divisors $P_i-P_j$. Since $\Pic(X)_2\cong H^1(X,\Z/2\Z)$,
the action of $\pi_1(B_{\mathrm{sm}})$ on $H^1(X,\Z/2\Z)$ is
completely determined by the permutations of the Weierstrass points
\lst P{2g+2} it induces. Thus, order of $\Mon(\X,\Z/2\Z)$ is at most
$(2g+2)!$. Since
\[
(\Sp(2g,\Z/2\Z):1)= 2^{g^2}(2^{2g}-1)(2^{2(g-1)}-1)\cdot\dots\cdot
(2^2-1), 
\]
the proposition follows.
\end{proof}

\begin{note}
  The bound in Proposition~\ref{mono:hyper} is sharp, which follows
  from A'Campo's paper~\cite{Acampo}. To wit, for any $g\ge2$ let us
  regard $\Aa^{2g+1}$ as the space of polynomials
\[
P(x)=x^{2g+2}+a_{2g}x^{2g}+\dots+a_1x+a_0,
\]
and denote the space of polynomials~$P\in\Aa^{2g+1}$ with a multiple
root by $\Sigma\subset\Aa^{2g+1}$. If \X is a family over
$\Aa^{2g+1}\setminus\Sigma$ in which the fiber over~$P$ is the smooth
projective model of the curve with equation $y^2=P(x)$ (which is
hyperelliptic of genus~$g$), then part of the corollary on page~319
of~\cite{Acampo} can be restated to the effect that index
$(\Sp(2g,\Z):\Mon(\X))$ is equal to the right-hand side
of~\eqref{eq:h-ell2}. Actually, not only the order of $\Mon(\X)$ is
known: a description of this group can be found in the appendix
to~\cite{Chmutov}, which (the appendix) is devoted to the exposition
of results of A.Varchenko.
\end{note}

\section{Appendix: an application to Del Pezzo surfaces}\label{sec3}

\epigraph{%
Robinson said, `It was only to be expected.'
}{%
--Muriel Spark, \emph{Robinson}}

\noindent
In this section, by way of an application of
Proposition~\ref{numcompts}, we prove the following fact.

\begin{proposition}\label{delPezzo}
If $X\subset\PP^n$ is a Del Pezzo surface embedded by \textup(a
subsystem of\textup) the anticanonical linear system~$|-K_X|$, then
the monodromy group acting on $H^1(\cdot,\Z)$ of its smooth hyperplane
sections is the entire $\SL(2,\Z)$.
\end{proposition}

First recall some notation and definitions.

If $X$ is an algebraic variety and \Rs is a coherent sheaf of reduced
$\Ooo_X$-algebras, we denote its relative spectrum (which is a scheme
over $X$) by $\Specb\Rs$ (under our assumptions $\Specb\Rs$ is an
algebraic variety and the canonical morphism $\Specb\Rs\to X$ is
finite).

If $p\in\PP^n$ is a point and $L\subset\PP^n$ is a linear subspace,
then $\overline{p,L}$ denotes the linear span of $\{p\}\cup L$.

If \lst A4 are points on the affine line with coordinates \lst a4,
then by their cross-ratio we mean
\[
[A_1,A_2,A_3,A_4]=\frac{a_3-a_1}{a_3-a_2}\bigg/\frac{a_4-a_1}{a_4-a_2}.
\]

If $X\subset\PP^n$ is a smooth
projective variety and $X^*\subset(\PP^n)^*$ is its projective dual,
one can define the ``universal smooth hyperplane section of $X$'',
that is, the family
\begin{equation}\label{eq:univ.hyper}
\U_X=\{(x,\alpha)\in X\times((\PP^n)^*\setminus X^*)\colon x\in H_\alpha\},
\end{equation}
where $H_\alpha\subset\PP^n$ is the hyperplane corresponding to the
point $\alpha\in(\PP^n)^*$. The morphism $\pi\colon
(x,\alpha)\mapsto\alpha$ makes \X a smooth family of $n$-dimensional
projective varieties over $(\PP^n)^*\setminus X^*$; for any natural
$d$, this family induces a monodromy action of
$\pi_1((\PP^n)^*\setminus X^*)$ on $H^d(Y,\Z)$, where $Y$ is a smooth
hyperplane section of~$X$.

In the above setting, the image of $\pi_1((\PP^n)^*\setminus X^*)$ in
the group $\Aut(H^n(Y,\Z))$ will be called \emph{hyperplane monodromy
  group} of~$X$.

\begin{lemma}\label{iso_proj}
Suppose that $X\subset\PP^n$ is a smooth projective variety and that
$p\in\PP^n\setminus X$ is a point such that the projection with center
$p$ induces an isomorphism $\pi_p\colon X\to X'\subset\PP^{N-1}$. If
$H\ni p$ is a hyperplane that is transversal to $X$, then, after
identifying $Y=X\cap H$ with $Y'=\pi_p(Y)=X'\cap\pi_p(H)$, the
hyperplane monodromy groups acting on $H^n(Y,\Z)$ and $H^n(Y',\Z)$,
are the same.
\end{lemma}

The proof that is sketched below was suggested to me by Jason Starr.

\begin{proof}[Sketch of proof]
Denote by $H_p\subset(\PP^n)^*$ the hyperplane corresponding to the
point $p\in\PP^n$. It is clear that $H_p$ is naturally isomorphic to
$(\PP^{N-1})^*$ and that $(X')^*=X^*\cap H_p$. Moreover, the
hyperplane $H_p$ is transversal to $X^*$ at any smooth point of $X^*$
(indeed, if $H_p$ is tangent to $X^*$ at a smooth point, then $p\in
(X^*)^*=X$, which contradicts the hypothesis).

To prove the lemma it suffices to show that
$\pi_1(H_p\setminus(X')^*)$ surjects onto $\pi_1((\PP^n)^*\setminus
X^*)$. To that end observe that there exists a line $\ell\subset H_p$
that is transversal to the smooth part of $X^*\cap H_p=(X')^*$ (in
particular, $\ell$ does not pass through singular points of $X^*\cap
H_p$). It follows from the transversality of $H_p$ to the smooth part
of $X^*$ that $\ell$ is transversal to the smooth part of $X^*$,
too. Thus, $\pi_1(\ell\setminus X^*)$ surjects both onto
$\pi_1(H_p\setminus(X')^*)$ and onto $\pi_1((\PP^n)^*\setminus X^*)$,
whence the desired surjectivity.
\end{proof}

Lemma~\ref{iso_proj} implies that when studying hyperplane monodromy
groups one may always assume that the variety in question is embedded
by a complete linear system. Recall that if a Del Pezzo
surface~$X\subset\PP^n$ is embedded by the complete linear system
$|-K_X|$ then $\deg X=n\le9$; besides, if $n>3$, $p\in X$ is a general
point, and $\bar X$ is the blow-up of $X$ at $p$, then the projection
$\pi_p\colon X\dasharrow \PP^{n-1}$ induces an isomorphism
$\bar\pi_p\colon \bar X\to X'=\overline{\pi_p(X)}\subset \PP^{n-1}$
and $X'\subset \PP^{n-1}$ is a Del Pezzo surface embedded by
$|-K_{X'}|$.

\begin{lemma}\label{n->3}
In the above setting, suppose that the hyperplane monodromy group of
$X'$ is the entire $\SL(2,\Z)$. Then the hyperplane monodromy group of
$X$ is the entire $\SL(2,\Z)$ as well.
\end{lemma}

\begin{proof}
Informally the proof may be summed up in one phrase: if variation of
hyperplanes passing through $p$ and transversal to $X$ is enough to
obtain the entire group $\SL(2,\Z)$, then \emph{a fortiori} this is
the case for all hyperplanes transversal to~$X$. A~formal argument
follows.

Assume that the $\PP^{n-1}$ into which the surface $X$ is projected is
a hyperplane in $\PP^n$, $\PP^{n-1}\not\ni x$. If a hyperplane
$H\subset\PP^n$ contains the point $p\in X$, then $H=\overline{p,h}$,
where $h=H\cap\PP^{n-1}$ is a hyperplane in $\PP^{n-1}$. If $H$ is
transversal to $X$, then $\pi_p$ induces an isomorphism between the
curve $H\cap X$, which is a smooth hyperplane section of $X$, and the
curve $h\cap X'$, which is a smooth hyperplane section of $X'$. Put
\[
V=\{\alpha\in (\PP^{n-1})^*\colon \text{$\overline{p,h_\alpha}$ is
  transversal to $X$}\},
\]
where $h_\alpha\subset\PP^{n-1}$ is the hyperplane corresponding to
the point $\alpha\in (\PP^{n-1})^*$, and set
\[
\U=\{(\alpha,x)\in V\times X'\colon x\in h_\alpha\}.
\]
In the diagram
\begin{equation}\label{diag2x3}
\xymatrix{
{\U_{X'}}\ar[r]\ar[d]^{q'}&{\U}\ar[r]\ar[d]&{\U_X}\ar[d]^q\\
{(\PP^{n-1})^*\setminus(X')^*}&V\ar[r]^(0.3)r\ar@{_{(}->}[l]_(0.25)j\ar[r]&
  {(\PP^n)^*\setminus X^*}
}
\end{equation}
where $\U_{X'}$ and $\U_X$ are universal smooth hyperplane sections of
$X'$ and $X$, $j$ is an open embedding, and $r$ maps a hyperplane
$h_\alpha\subset \PP^{n-1}$ to the hyperplane $\overline{p,h_\alpha}$
in $\PP^n$, both squares are Cartesian. Pick a point $\alpha\in V$;
the hyperplane $\overline{p,h_\alpha}\subset\PP^n$ is $H_{r(\alpha)}$,
where $r(\alpha)\in(\PP^n)^*$. If $Y'={q'}^{-1}(\alpha)$ and
$Y=X\cap H_{r(\alpha)}=q^{-1}(r(\alpha))$, then in the
commutative diagram
\[
\xymatrix{
{\pi_1((\PP^{n-1})^*\setminus (X')^*,\alpha)}\ar[dr]_u&
    {\pi_1(V,\alpha)}\ar[r]\ar[d]\ar[l]_(.35){w}&{\pi_1((\PP^n)^*\setminus
      X^*,r(\alpha))}\ar[d]^v\\
 & {\Aut(H^1(Y',\Z))}\ar@{=}[r]&{\Aut(H^1(Y,\Z))}
}
\]
the mapping $w$ is an epimorphism since $V$ is Zariski open in
$(\PP^{n-1})^*\setminus (X')^*$, whence $\im u\subset\im v$. This
proves the lemma.
\end{proof}

Projecting Del Pezzo surfaces in $\PP^n$, $n>3$, consecutively from
general points on them, one arrives at a cubic in $\PP^3$;
Lemma~\ref{n->3} implies that it suffices to prove Proposition~\ref{delPezzo}
for this surface.

The next lemma reduces the problem to the case of ``Del Pezzo surfaces of
degree~$2$''.

Suppose that $X\subset\PP^3$ is a smooth cubic and $p\in X$ is a
general point. Let $\bar X$ be the blow-up of $X$ at $p$.  Then the
projection $\pi_p\colon \PP^3\dasharrow \PP^2$ induces a finite
morphism $\bar\pi_p\colon\bar X\to\PP^2$ of degree~$2$; the branch
locus of this morphism is a smooth curve $C\subset\PP^2$ of
degree~$4$. For $\alpha\in (\PP^2)^*$, denote the corresponding line
by $\ell_\alpha\subset \PP^2$. If $\ell_\alpha$ is transversal to $C$
(i.e., $\alpha\notin C^*$), then $\bar\pi_p^{-1}(\ell_\alpha)$ is
smooth, irreducible, and isomorphic to $X\cap
\overline{p,\ell_\alpha}$.

The proof of the following lemma is similar to that of Lemma~\ref{n->3}.
\begin{lemma}\label{3->2}
Put 
\begin{equation}\label{eq:def_of_X}
\X=\{(\alpha,x)\in((\PP^2)^*\setminus C^*)\times \bar X\colon
\bar\pi_p(x)\in\ell_\alpha\} 
\end{equation}
and denote the morphism $(\alpha,x)\mapsto\alpha$ by $q\colon \X\to
(\PP^2)^*\setminus C^*$. If $\Mon(\X,\Z)=\SL(2,\Z)$, then the
hyperplane monodromy group of $X$ is also equal to $\SL(2,\Z)$.
\end{lemma}

Our next lemma is valid over algebraically closed fields
of arbitrary characteristic. 

\begin{lemma}\label{lemma:Stein}
Suppose that $W$ is a smooth irreducible variety of dimension~$n$, $L$
is a smooth irreducible curve \textup(we do not assume that $W$ or $L$ is
projective\textup), and $\ph\colon W\to L$ is a proper and surjective morphism
with $(n-1)$-dimensional fibers. Put $Z=\Specb \ph_*\Ooo_W$ and let
$v\colon Z\to L$ be the natural morphism.

If there exists a point $p\in L$ such that $\ph^{-1}(p)$ is
irreducible and the morphism $\ph$ has maximal rank at a general point
of $\ph^{-1}(p)$, then the natural morphism $v\colon Z\to L$ is an
isomorphism.
\end{lemma}

\begin{proof}
It is clear that $Z$ is an irreducible and reduced curve.  Since $\ph$
is proper and $\ph^{-1}(p)$ is connected, the stalk $(\ph_*\Ooo_W)_p$
is a local ring, so $v^{-1}(p)$ consists of one point; denote this
point by~$z$.  I claim that that $z$ is a smooth point of~$Z$ and the
morphism $v$ is unramified at~$z$. Indeed, let $\tau\in\Ooo_{L,p}$ be
a generator of the maximal ideal. Its image $v^*\tau\in\Ooo_{Z,z}$ can
be represented by a regular function $f\in\Ooo_W(\ph^{-1}(U))$, where
$U\subset L$ is a Zariski neighborhood of~$p$. Since the morphism
$\ph$ has maximal rank at a general point of $\ph^{-1}(p)$, the
function $v^*\tau$ vanishes on the irreducible divisor $\ph^{-1}(p)$
with multiplicity~$1$. Since regular functions on $\ph^{-1}(U)$ must
be constant on the fibers of the proper morphism~$\ph$, any element of
the maximal ideal of the local ring $\Ooo_{Z,z}$ is representable by a
regular function $g\in\Ooo_W(\ph^{-1}(V))$, where $V$ is a Zariski
neighborhood of $p$, such that the zero locus of $g$ in $\ph^{-1}(V)$
coincides with $u^{-1}(z)$. Hence, $v^*\tau$ generates the maximal
ideal of $\Ooo_{Z,z}$, which proves our claim.

Since $v^{-1}(p)=\{z\}$, $Z$ is smooth at $z$, and $v$ is unramified
at $z$, we conclude that the finite morphism $v$ has degree~$1$. Since
$L$ is smooth, Zariski main theorem implies that $v$ is an
isomorphism.
\end{proof}

\begin{proposition}\label{j-fibers}
Suppose that $\pi\colon X\to\PP^2$ is a finite morphism of degree $2$
branched over a smooth quartic $C\subset \PP^2$, where $X$ is
smooth. If $J\colon (\PP^2)^*\setminus C^*\to\Aa^1$ is the morphism
$\alpha\mapsto j(\pi^{-1}(\ell_\alpha))$, where $\ell_\alpha$ is the
line in $\PP^2$ corresponding to $\alpha\in(\PP^2)^*$, then a general
fiber of $J$ is irreducible.
\end{proposition}

\begin{proof}
Let us show that the morphism $J$ extends to a morphism
\[
J_1\colon (\PP^2)^*\setminus(C^*)_{\mathrm{sing}}\to\PP^1=\Aa^1\cup\{\infty\}.
\]
Indeed, if $\ell\subset\PP^2$ is a line and $\ell\cap
C=\{P_1,P_2,P_3,P_4\}$, then the curve $\pi^{-1}(\ell)$ is a curve of
genus~$1$ and
\begin{equation}\label{j:formula}
j(\pi^{-1}(\ell))=256\frac{(\lambda^2-\lambda+1)^3}{\lambda^2(1-\lambda)^2},
\end{equation}
where $\lambda$ is the cross-ratio~$[P_1,P_2,P_3,P_4]$, in no matter
what order (see for example~\cite[Chapter~III,
  Proposition~1.7b]{Silverman}). If $\alpha$ is a smooth point of
$C^*\subset(\PP^2)^*$, then the line $\ell_\alpha$ is tangent to $C$
at exactly one point that is not an inflection point. Thus, as the
line $\ell$ tends to $\ell_\alpha$, exactly two intersection points
from $\ell\cap C$ merge, so the cross-ratio of these four points
tends to $0$ (or~$1$, or~$\infty$, depending on the ordering), and
formula~\eqref{j:formula} shows that $j(\pi^{-1}(\ell))$ tends
to~$\infty$. This proves the existence of the desired extension.

Our argument shows that
$J_1^{-1}(\infty)=C^*\setminus(C^*)_{\mathrm{sing}}$; if we regard
$J_1$ as a rational mapping from $(\PP^2)^*$ to $\PP^1$ and if
\begin{equation}\label{diag:resolution}
\xymatrix{
{W}\ar[dr]^{J_2}\ar[d]^\sigma\\
{(\PP^2)^*}\ar@{-->}[r]^{J_1}&{\PP^1}
}
\end{equation}
is a minimal resolution of indeterminacy for $J_1$, then
$J_2^{-1}(\infty)$ equals the strict transform of
$C^*$ with respect to~$\sigma$.

Now I claim that, at a general point of $J_2^{-1}(\infty)$, derivative
of $J_2$ has rank~$1$. It suffices to prove this assertion for $J_1$
and a general smooth point of $C^*$. To that end it suffices to
construct an analytic mapping $\gamma\colon D\to(\PP^2)^*$, where $D$
is a disk in the complex plane with center at~$0$, such that
$\gamma(D\setminus\{0\})\subset(\PP^2)^*\setminus C^*$, $\gamma(0)$ is
a smooth point of $C^*$, and
$|j(\pi^{-1}(\ell_{\gamma(t)}))|\sim\mathrm{const}/|t|$.

Suppose that a point $c\in C$ is not an inflection point nor a
tangency point of a bitangent; if $\ell_\alpha\subset\PP^2$ is the
tangent line to $C$ at~$c$, then $\alpha$ s a smooth point of
$C^*$. Now choose affine $(x,y)$-coordinates in $\PP^2$ so that
$c=(0,0)$, the tangent $\ell_\alpha$ has equation $y=0$, and
$\ell_\alpha\cap C=\{c,(C,0),(D,0)\}$, where $C,D\ne 0$ (so the
remaining two points of $\ell_\alpha\cap C$ are in the finite part of
$\PP^2$ with respect to the chosen coordinate system). If
$\ell_{\gamma(t)}$ is the line with affine equation $y=t$, then, for
all small enough~$t$, one has $\ell_{\gamma(t)}\cap
C=\{A(t),B(t),C(t),D(t)\}$, where the $x$-coordinates of $A(t)$ and
$B(t)$ are $\sqrt t+o(\sqrt{|t|})$ (for both values of $\sqrt t$),
while the $x$ coordinates of $C(t)$ and $D(t)$ tend to finite and
non-zero numbers~$C$ and~$D$. Hence,
\[
\big|[C(t),A(t),B(t),D(t)]\big|\sim
\frac{\text{const}}{\sqrt{|t|}}\quad\text{as $t\to0$};
\]
formula~\eqref{j:formula} implies that
$|j(\pi^{-1}(\ell_t))|\sim\mathrm{const}/|t|$, as desired.

Let
\[
\xymatrix{
{W}\ar[rr]^{J_2}\ar[dr]_u&&{\PP^1}\\
&Z\ar[ur]_v
}
\]
be the Stein factorization in which $W$ is a blow-up of $(\PP^2)^*$
(see~\eqref{diag:resolution}), $Z=\Specb (J_2)_*\Ooo_{W}$, and $v$ is
a finite morphism. Applying Lemma~\ref{lemma:Stein} with $L=\PP^1$,
$\ph=J_2$, and $p=\infty$, we conclude that $v$ is an
isomorphism. Thus, fibers of $J_2$ coincide with fibers of~$u$; since
the latter are connected, fibers of $J_2$ are connected as well.
Bertini theorem implies that a general fiber of $J_2$ is smooth; since
it is connected, it must be irreducible. This implies that a general
fiber of $J$ is irreducible.
\end{proof}

\begin{proof}[Proof of Proposition~\ref{delPezzo}]
In view of Proposition~\ref{iso_proj} and Lemmas~\ref{n->3}
and~\ref{3->2}, it suffices to prove that $\Mon(\X)=\SL(2,\Z)$, where
\X is the family defined by~\eqref{eq:def_of_X}.

Applying Proposition~\ref{j-fibers} to the surface $\bar X$ (blow-up
of a cubic at a general point~$p$) and the mapping
$\bar\pi_p\colon\bar X\to\PP^2$ (induced by the projection with
center~$p$), we see that the family \X defined by
formula~\eqref{eq:def_of_X} satisfies the hypothesis of
Proposition~\ref{numcompts}(i), whence $\Mon(\X)=\SL(2,\Z)$.
\end{proof}

\begin{note}
Our argument shows as well that if $X$ is a Del Pezzo surface of
degree $2$, then the monodromy group acting on $H^1(\cdot,\Z)$ of
non-singular elements of the anticanonical linear system $|-K_X|$, is
$\SL(2,\Z)$. I do not know the answer for Del Pezzo surfaces of
degree~$1$. 
\end{note}

\bibliographystyle{amsplain}

\bibliography{bib}

\end{document}